\documentclass[11pt]{amsart}
\usepackage{amsthm}
\usepackage{amsmath}
\usepackage{amssymb}
\usepackage{enumerate}
\usepackage{tikz}

\newtheorem{theorem}{Theorem}[section]
\newtheorem{lemma}[theorem]{Lemma}

\newtheorem{corollary}[theorem]{Corollary}
\newtheorem{definition}[theorem]{Definition}

\newcommand \F {\mathbb F}

\newcommand \irr {\textup{Irr}}

\newcommand \ibr {{\textup{IBr}}_p}

\newcommand \opd {{{\bf{O}}_{p'}}}
\newcommand \op {{{\bf{O}}_{p}}}

\newcommand{\dl}{\mathrm {dl}}
\newcommand{\cd}{{\mathrm {cd}}}
\newcommand \dvd {\hbox {\big|}}
\newcommand \ndvd {\hbox {/}\kern-5pt\dvd}
\newcommand \nrml {\lhd}

\newcommand \cent {{\bf{C}}}
\def \< {\langle}
\def \> {\rangle}

\begin{document}

\title{A height-zero type result for blocks of solvable groups}

\author{James P. Cossey}

\address{Department of Mathematics, University of Akron, Akron, OH 44325}

\email{cossey@uakron.edu}

\keywords{Brauer character, Finite groups, Representations, Solvable groups}
\subjclass[2000]{Primary 20C20}

\begin{abstract}  Let $B$ be a $p$-block of a finite group $G$ with defect group $D$.  The more difficult direction of the recently proven height zero conjecture says that $D$ is abelian if every character in $\irr(B)$ has height zero.  We consider a smaller set than $\irr(B)$.  In particular, if $\varphi \in \ibr(B)$, we let $\irr(\varphi)$ be the set of characters $\chi \in \irr(G)$ such that $\varphi$ is a constituent of $\chi^o$.   Now suppose $G$ is solvable and $\varphi$ is a height zero Brauer character in some block $B$ of $G$ with defect group $D$.  Here we show that if every character in $\irr(\varphi)$ has height zero, then the defect group $D$ of the block containing $\varphi$ is abelian for $p \geq 5$ and almost abelian for $p = 2$ or $3$.  This has a nice consequence for primitive characters of $p$-complements in solvable groups.
  
\end{abstract}

\maketitle

\section{Introduction} 

Let $G$ be a finite group and $p$ a prime.  The Ito-Michler theorem \cite{itomichler} tells us that $G$ has a normal abelian Sylow $p$-subgroup if and only if $p \ndvd \chi(1)$ for every character $\chi \in \irr(G)$.  Now define $\irr_p(G)$ to be the set of $\chi \in \irr(G)$ that have degree divisible by $p$.  Thus, one can state the difficult direction of the Ito-Michler theorem as saying that if $|\irr_p(G)| = 0$, then $G$ has a normal, abelian Sylow $p$-subgroup.

One can get a similar result by looking at a set smaller than $\irr(G)$.  If $B$ is a block of $G$ with defect group $D$, then the (recently proven, see \cite{heightzero}) height zero conjecture proves that $D$ is abelian if and only if every character in $\irr(B)$ has height zero.  If we define $\irr_p(B)$ as being the set of $\chi \in \irr(B)$ such that $p$ has positive height, then we can restate the more difficult direction of the height zero conjecture as saying that if $|\irr_p(B)| = 0$, then $D$ is abelian.

In this paper we show, for solvable groups, that one can look at an even smaller subset of $\irr(G)$ to determine if  the defect group of the associated block is abelian, or at least almost abelian (in the sense of having a small derived length).  Before we state our main result, we need some definitions.

Recall (see for instance \cite{navarrobook}) that if $\chi \in \irr(G)$, then there is a unique decomposition $$\chi^o = \sum_{\varphi \in \ibr(G)} d_{\chi \varphi} \varphi,$$ where $\chi^o$ denotes the restriction of $\chi$ to the $p$-regular elements of $G$.  Also, we let $h(\chi)$ denote the height of the character $\chi$.

\begin{definition}  Fix a Brauer character $\varphi$ in a $p$-block $B$ of a group $G$.  We define the set $$\irr(\varphi) = \{ \chi \in \irr(G) | d_{\chi \varphi} \neq 0 \}.$$  We define $$\irr_+(\varphi) = \{ \chi \in \irr(\varphi) | h(\chi) > 0  \}.$$    Finally, we define  $k_+(\varphi) = |\irr_+(\varphi)|.$
\end{definition}

Of course, another way to put this is that $\irr(\varphi)$ consists of the irreducible constituents of the projective indecomposable character for $\varphi$.  Our first main result is the following:

\begin{theorem}\label{main1}  Let $B$ be a $p$-block of a solvable group $G$ with defect group $D$.  Suppose that $\varphi \in \ibr(B)$ has height zero and that $k_+(\varphi) = 0$.  If $p \geq 5$, then $D$ is abelian.  If $p = 2$, then $\dl(D) \leq 2$, and if $p = 3$, then $\dl(D) \leq 3$.
\end{theorem}

We will in fact prove a generalization of this result, in terms of the positive height characters in $\irr(\varphi)$ - where $\varphi$ has height zero - and the derived length of the defect group of the block containing $\varphi$.  

Theorem \ref{main1} has an interesting corollary, which can be stated independently of block theory.

\begin{corollary}\label{primcor}  Let $H$ be a $p$-complement of a solvable group $G$.  Let $\alpha \in \irr(H)$ be primitive, and suppose that every constituent of $\alpha^G$ has $p'$-degree.  If $S$ is a Sylow $p$-subgroup of $G$, then $S$ is abelian if $p \geq 5$.  If $p =2$, then $\dl(S) \leq 2$.  If $p = 3$, then $\dl(S) \leq 3$.
\end{corollary}

We will actually prove slightly more here as well, replacing \lq \lq $\alpha$ is primitive" with \lq \lq $\alpha$ is a Fong character".

Finally, we mention that our main tool in proving these results is the following \lq \lq large orbit" theorem, which may prove to have other applications.

\begin{theorem}\label{orbits}  Let $G$ be a solvable group that acts faithfully and completely reducibly on a finite module $V$ of characteristic $p$.  Suppose that $$p \ndvd |G : \cent_G(v)|$$ for every $v \in V$.   Let $S$ be a Sylow $p$-subgroup of $G$.  If $p \geq 5$, then $S$ is trivial.  If $p = 2$, then $S$ is abelian.  Finally, if $p=3$, then $\dl(S) \leq 2$.
\end{theorem}

As we will see, the proof of Theorem \ref{orbits} follows easily from the work of Wolf, Gluck, and others.

\section{The proof of Theorem \ref{orbits}}

Before we can prove Theorem \ref{orbits}, we need a result on semilinear groups.  Let $p$ be a prime and $n \geq 1$.  We let $V$ be the vector space of dimension $n$ over the field $\F_p$ of $p$ elements.  We may identify $V$ with the field $\F_{p^n}$ of $p^n$ elements in the natural way.  The semilinear group $\Gamma(p^n)$ has a normal subgroup $\Gamma_K$, where $\Gamma_K$ acts on $V$ via field multiplication, and $\Gamma(p^n)/\Gamma_K$ acts on $V$ as the Galois group $Gal(\F_{p^n}/\F_p)$.   If $G$ is a subgroup of $\Gamma(p^n)$, then we say $G$ is semilinear, with the assumption that $G$ acts on $V$ in the same way.  Thus $G$ has a normal subgroup $K$ such that $K$ acts on $V$ as field multiplication, and $G/K$ acts on $V$ as a subgroup of the Galois group $Gal(\F_{p^n}/\F_p)$.  See \cite{manzwolf} for more details.

\begin{lemma}\label{semilinearaction}  Let $G$ be a semilinear group, acting on $V = \F_p^n$ in the manner described above.  Suppose that for every $v \in V$, we have that $p \ndvd |G : \cent_G(v)|$.  Then $p \ndvd |G|$.
\end{lemma}

\begin{proof}  We let $K \nrml G$ be the subgroup of $G$ that acts as multiplication on $\F_{p^n}$, and thus $G/K$ is a subgroup of $Gal(\F_{p^n}/\F_p)$.  Write ${\overline{G}} = G/K$.

Notice that $K$ acts regularly on $V$.  Thus for any $v \in V$, we have that $$K \cap \cent_G(v) = 1.$$  Thus we may consider $\cent_G(v)$ to be a subgroup of ${\overline{G}}$.  By assumption, $\cent_G(v)$ contains the unique Sylow $p$-subgroup of ${\overline{G}}$.  As this is true for every $v \in V$, then the Sylow $p$-subgroup of $\overline{G}$ fixes all of $V \cong \F_{p^n}$, and therefore is trivial.
\end{proof}

We now prove Theorem 1.4, which should be considered a \lq \lq large orbit theorem": if no orbit is \lq \lq large enough" to have size divisible by $p$, then the $p$-part of the group must be \lq \lq small". 

\begin{proof}

We work by induction on $|G| + |V|$.  First, suppose the action of $G$ on $V$ is reducible.  Then we may write $$V = V_1 \oplus V_2 \oplus \ldots V_k$$ for the homogeneous constituents of $V$, and it follows that we may write $$G = G_1 \times G_2 \times \ldots G_k,$$ where $G_i$ acts faithfully and completely reducibly on $V_i$.  For $v_i \in V_i$, we have that $|G_i : \cent_{G_i}(v_i)|$ divides $|G : \cent_G(v_i)|$, and thus $p \ndvd |G_i : \cent_{G_i}(v_i)|$.  If $S_i$ is a Sylow $p$-subgroup of $G_i$, then the inductive hypothesis tells us that $S_i$ is trivial if $p \geq 5$, $S_i$ is abelian if $p =2$, and $\dl(S_i) \leq 2$ if $p = 3$.  Thus we are done in this case.

We now have that $V$ is irreducible, and we assume that $V$ is quasiprimitive.  By Theorem 2.4 of \cite{yanghuppert} we have that either $G$ is semilinear, or $p = 3$ and $G = GL(2, 3)$ or $SL(2, 3)$.  If $G$ is semilinear, then Lemma \ref{semilinearaction} tells us that $p \ndvd |G|$.  In either case we are done.

Thus we now assume that $G$ is not quasiprimitive on $V$.  Let $C \nrml G$ be maximal such that $V_C$ is not homogeneous.  Of course, $C < G$.  Then (see Lemma 0.2 of \cite{manzwolf}) $G/C$ faithfully and primitively permutes the homogeneous constituents $W_1, W_2, \ldots W_{\ell}$ of $V_C$.  Write $L/C = {\bf{O}}^{p'}(G/C)$.  If $G/L$ is nontrivial, then $V_L$ is a faithful, completely reducible $L$-module, and $$p \ndvd |L : \cent_L(v)|$$ for every $v \in V$.  As any Sylow $p$-subgroup of $L$ is a Sylow $p$-subgroup of $G$, we are done if $L < G$.

Thus we have that $L = G$.  As $C < G$, it follows that $p \dvd |G : C|$.  By Theorem 9.3 of \cite{manzwolf}, it follows that $p = 2$ or $3$, and the $p$-part of $|G : C|$ is either $2$ or $3$.  (Note we are now done if $p \geq 5$.)  Thus $G/C$ has an abelian Sylow $p$-subgroup.  Moreover, by Theorem 9.3 of \cite{manzwolf}, we have that $C/\cent_C(W_i)$ acts transitively on the nonzero elements of $W_i$.  By a theorem of Huppert (see for instance Theorem 6.8 of \cite{manzwolf}), $C/\cent_C(W_i)$ is either semilinear or one of a very short list of exceptional cases.  If $C/\cent_C(W_i)$ is semilinear, then Lemma \ref{semilinearaction} tells us a Sylow $p$-subgroup of $C/\cent_C(W_i)$ is trivial.  If $C/\cent_C(W_i)$ is one of the exceptional cases in Huppert's theorem, then either $p \ndvd |C/\cent_C(W_i)|$, or $p = 3$ and $|C/\cent_C(W_i)|_p = 3$, and thus a Sylow $p$-subgroup of $C/\cent_C(W_i)$ is abelian.  In this case we see that $C$ has an abelian Sylow $p$-subgroup.  In all cases, then, we are done. 
\end{proof}

\section{Bounding the derived length of $D$}

We need an easy lemma regarding characters in $\irr(\varphi)$.

\begin{lemma}\label{overlemma}  Let $G$ be solvable and $N \nrml G$.  Let $\varphi$ be in  $\ibr(G)$ and suppose $\theta \in \ibr(N)$ lies under $\varphi$.  If $\psi \in \irr(\theta)$, then some constituent of $\psi^G$ is in $\irr(\varphi)$.
\end{lemma}

\begin{proof}  By assumption, $\theta$ is a constituent of $\psi^o$.  Also, $\varphi$ is a constituent of $\theta^G$.  Thus $\varphi$ is a constituent of $(\psi^G)^o = (\psi^o)^G$, and therefore some constituent $\chi$ of $\psi^G$ must contain $\varphi$. 
\end{proof}

We here prove a more general version of Theorem \ref{main1}.   The following result includes Theorem \ref{main1}, and the proof very much mirrors the argument in \cite{4paper}.

\begin{theorem}\label{main2}  Let $B$ be a block of a solvable group $G$ with defect group $D$.  Let $\varphi \in \ibr(B)$ have height zero.  Then $$\dl(D) \leq k_+(\varphi) + c,$$ where $c = 1$ if $p \geq 5$, $c = 2$ if $p =2$, and $c = 3$ if $p = 3$.
\end{theorem}

\begin{proof}   Let $N = \opd(G)$, and consider a counterexample with $|G : N|$ minimal.  Suppose $\alpha \in \irr(N)$ is covered by $B$, and let $T$ be the stabilizer of $\alpha$ in $G$.  By the Fong-Reynolds theorem (see Theorem 9.14 of \cite{navarrobook}, for example), there is a block $B_1$ of $T$ with defect group conjugate to $D$ such that induction is a height preserving bijection from $B_1$ to $B$ that also preserves decomposition numbers.  

Thus without loss of generality, we may assume $\alpha$ is invariant in $G$.  By Fong's theorem (see 10.20 of \cite{navarrobook}) we have that $D$ is a full Sylow $p$-subgroup of $G$, and $B$ consists of all of the characters (ordinary and Brauer) that lie over $\alpha$.  As $\varphi$ has height zero, it follows that $p \ndvd \varphi(1)$.


Moreover, by applying a character triple isomorphism, we may assume $N$ is central in $G$, and if we set $M/N = \op(G/N)$, then $${\bf{F}}(G) = M = N \times P$$ and ${\bf{F}}(G/N) = M/N \cong P$.

We first show that $P$ must be abelian.  Assume not.  By Lemma 5.1 of  \cite{4paper} we have that $NP'/P' = \opd(G/P')$.   Let ${\widehat{\alpha}}$ denote the unique $p'$-special extension of $\alpha$ to $NP'$, so that we may consider ${\widehat{\alpha}}$ to be a character of $NP'/P'$.  Note that ${\widehat{\alpha}}$ is invariant in $G/P'$.  We let $b'$ be the block of $G/P'$ lying over ${\widehat{\alpha}}$.  As $P'$ is in the kernel of $\varphi$, we may consider $\varphi$ as a character in $\ibr(b')$.  Note that $p \ndvd \varphi(1)$, and thus $D/P'$ is a defect group of $b'$.  We let $k_+'(\varphi)$ denote the number of characters of positive height in $\irr(\varphi) \cap \irr(G/P')$.  As $G$ was chosen so that $|G/N|$ is minimal, we see that $$dl(D/P') \leq  k_+'(\varphi) + c.$$

Let $\cd_p(P)$ be the number of distinct nonlinear character degrees of $P$.  By Taketa's theorem, we have that $$\dl(P) - 1 \leq \cd_p(P).$$  Suppose $\beta \in \irr(P)$ is nonlinear.  By Lemma \ref{overlemma}, we know there exists a character $\chi \in \irr(\varphi)$ lying above $\alpha \times \beta \in \irr(M)$.  Of course, $p \dvd \chi(1)$, so $\chi$ is counted by $k_+(\varphi)$, but $\chi$ is not counted by $k_+'(\varphi)$.  Thus there are at least $\cd_p(P)$ characters counted by $k_+(\varphi)$ that are not counted by $k_+'(\varphi)$.

We now have $$k_+(\varphi) \geq k_+'(\varphi) + \cd_p(P) \geq \dl(D/P') - c + \dl(P) - 1,$$ which in turn is equal to $$\dl(D/P') - c + \dl(P')  \geq \dl(D) - c,$$ and we are done if $P$ is not abelian. 

Thus we now assume that $P$ is abelian.  Consider the block $\bar{B}$ of $\bar{G} = G/P$ containing $\varphi$.  As $\varphi$ has $p'$-degree, then $D/P$ is a defect group for ${\bar{B}}$.  As $G$ was a minimal counterexample, we see that if ${\bar{k}}_+(\varphi)$ denotes the number of characters in $\irr_+(\varphi)$ with $P$ in the kernel, then $$\dl(D/P) \leq {\bar{k}}_+(\varphi) + c.$$  

Again, we have that $\dl(D) \leq \dl(D/P ) + 1$.  Thus, to show that $\dl(D) \leq k_+(\varphi) + c$, it is enough to show that there is a character in $\irr_+(\varphi)$ not counted by ${\bar{k}}_+(\varphi)$, i.e. that does not have $P$ in the kernel.

By Gaschutz' theorem, we see that $G/M$ acts faithfully and completely reducibly on $\irr(K)$, where $K$ is a nontrivial elementary abelian factor group of the $p$-group $M/N \cong P$.  By Lemma \ref{overlemma}, if $\lambda \in \irr(K)$, then there exists a character in $\irr(G | \alpha \times \lambda) \cap \irr(\varphi)$.  If every character in $\irr_+(\varphi)$ was counted by ${\bar{k}}_+(\varphi)$, then we would have $p \ndvd |G : G_{\lambda}|$ for every $\lambda \in \irr(K)$.  By Theorem \ref{orbits}, it then follows that $D/P$ has derived length zero if $p \geq 5$, at most one if $p = 2$, and at most two if $p = 3$.  As $P$ is abelian, it follows that $\dl(D) \leq 1$ if $p \geq 5$, $\dl(D) \leq 2$ if $p = 2$, and $\dl(D) \leq 3$ if $p = 3$.  But this contradicts that $G$ was a counterexample, and we are done. 
\end{proof}

Recall that if $G$ is solvable with $p$-complement $H$, and $\varphi \in \ibr(G)$, then a Fong character associated with $\varphi$ is any character of minimal degree in $\varphi_H$.  Fong characters have a number of interesting properties.  For instance, if $\alpha \in \irr(H)$ is a Fong character for $\varphi$, then $\alpha(1) = \varphi(1)_{p'}$.  Also, the constituents of $\alpha^G$ are exactly the characters in $\irr(\varphi)$.  See \cite{isaacsfong}

\begin{corollary}\label{fong}  Let $G$ be solvable and let $H$ be a $p$-complement in $G$.  Suppose that $\alpha \in \irr(H)$ is a Fong character for some character $\varphi \in \ibr(G)$, and that every constituent of $\alpha^G$ has $p'$-degree.  Let $S$ be a Sylow $p$-subgroup of $G$.  Then $S$ is abelian if $p \geq 5$, $\dl(S) \leq 2$ if $p = 2$, and $\dl(S) \leq 3$ if $p = 3$.
\end{corollary}

\begin{proof}  Our assumption that every constituent $\chi$ of $\alpha^G$ has $p'$-degree means that every character in $\irr(\varphi)$, including the canonical lift of $\varphi$, has $p'$-degree.  Thus $\varphi$ has $p'$-degree and height zero, and the defect group for the block $B$ containing $\varphi$ is a full Sylow $p$-subgroup of $G$.  By Theorem \ref{main1}, we get the desired bounds on the derived length of $S$.
\end{proof}

We mention another useful fact about Fong characters in solvable groups:  all primitive characters of the $p$-complement $H$ are Fong characters \cite{isaacsfong}.  Thus Corollary \ref{primcor} follows immediately from Corollary \ref{fong}.

\section{Examples and questions}

One might ask if we can do better than the statement of Theorem \ref{main1}.   Let $\varphi$ be the unique linear Brauer character in the unique $2$-block $B$ of $S_4$.  Each character in $\irr(\varphi)$ has height zero, yet the defect group for $B$ is the nonabelian Sylow $2$-subgroup of $S_4$.  Thus we cannot reduce our bound of $\dl(D) \leq 2$ for $p = 2$.

It is unclear if the bound of $\dl(D) =3$ can be achieved for $p = 3$.  If so, it likely comes from a block of a group $G$ defined in some sense by the exceptional case in Theorem 9.3 of \cite{manzwolf}.

It is also unclear to what extent this result extends to nonsolvable groups.  Of course, Wolf and Gluck extended their proof of the height zero conjecture for solvable groups to $p$-solvable groups in \cite{gluckwolf}.  Perhaps similar techniques could allow us to extend our result to $p$-solvable groups.

For symmetric or alternating groups, the path forward seems a bit more clear.  While the degrees of Brauer characters in symmetric groups are notoriously difficult to compute, it seems possible we may dispense with the requirement that $\varphi$ has height zero, and rather require that every character in $\irr(\varphi)$ has height zero.  As the heights of characters in $\irr(S_n)$ are easy to understand combinatorially, we then need only show that if all of the partitions that \lq \lq $p$-regularize" to $\varphi$ have height zero, then the weight of the associated block must be less than $p$.  This would force the corresponding defect group to be abelian.

While we are not comfortable conjecturing that a bound such as that in Theorem \ref{main1} holds for all groups, we think exploring this idea merits further study.  Specifically, what information (if any) does knowing $k_+(\varphi)$ tell us about the corresponding defect group for arbitrary finite groups?


\end{document}